\documentclass[oneside,reqno]{amsart}

\usepackage[english]{babel}
\usepackage[latin1]{inputenc}
\usepackage[sc]{mathpazo}
\usepackage[colorlinks]{hyperref}
\usepackage{enumitem}
\usepackage{pgfplots}
\pgfplotsset{compat=1.10}
\usepgfplotslibrary{fillbetween}
\usetikzlibrary{patterns}

\usepackage{todonotes}

\usepackage{amssymb,amsmath,amsthm,amsfonts, mathrsfs, mathtools}

\usepackage{enumitem}

\def\R {\mathbb{R}}
\def\N {\mathbb{N}}
\def\C {\mathcal{C}}
\def\eps{\varepsilon}

\DeclareMathOperator{\diam}{diam}

\newtheorem{proposition}{Proposition}[section]
\newtheorem{theorem}[proposition]{Theorem}
\newtheorem*{theorem*}{Theorem}

\newtheorem{lemma}[proposition]{Lemma}

\theoremstyle{definition}

\numberwithin{equation}{section}
\title[Predator-prey (III): classification]{Predator-prey models with competition, Part III: \\Classification of stationary solutions}
\author{Henri Berestycki}
\email{hb@ehess.fr}
\address{\'{E}cole des Hautes \'{E}tudes en Sciences Sociales, PSL Research University Paris, Centre d'analyse et de math\'{e}matique sociales (CAMS), CNRS, 54 bouvelard Raspail, 75006, Paris}

\author{Alessandro Zilio}
\email{azilio@math.univ-paris-diderot.fr}
\address{Universit\'{e} Paris Diderot - Paris 7, Laboratoire J.-L.\ Lions (CNRS UMR 7598), Paris, France, 8 place Aur\'elie Nemours, 75205, Paris CEDEX 13}


\begin{document}

\begin{abstract}
For a stationary system representing prey and $N$ groups of competing predators, we show classification results about the set of positive solutions. In particular, we show that if the number of components $N$ is too large or if the competition between different groups is too small, then the system has only constant solutions, which we then completely characterize.
\end{abstract}

\maketitle

\begin{center}
  \emph{To Luis Caffarelli, with admiration and affection.}
\end{center}

\section{Introduction}\label{sec intro}

In this article, we consider the set of classical non-negative solutions $\mathbf{v} = (\mathbf{w},u)$ of the following system of elliptic semilinear equations in a bounded smooth domain $\Omega \subset \R^n$, 
\begin{equation}\label{eqn model k same}
	\begin{cases}
		- d \Delta w_i  = \left(- \omega + k u - \beta \sum_{j \neq i} w_j\right) w_i &\text{in $\Omega$}\\
		- D \Delta u = \left(\lambda - \mu u - k \sum_{i=1}^N w_i \right)u &\text{in $\Omega$}\\
		\partial_\nu w_i = \partial_\nu u = 0 &\text{on $\partial \Omega$}.
	\end{cases}
\end{equation}
Here $N = \#\{i: w_i \geq 0 \text{ and } w_i \not \equiv 0\}$ stands for the number of non-zero components of the vector $\mathbf{w}$, and we denote $\mathbf{w} = (w_1,\dots, w_N)$.

This system models the interaction between a prey (spatially distributed as the density $u$) and $N$ groups of competing ($\beta > 0$) predators (the densities $w_i$) in an environment $\Omega\subset \R^n$. We recently introduced this model in \cite{BerestyckiZilio_PI, BerestyckiZilio_PII, BerestyckiZilio_Eco} with the aim to describe the ecological impact of territorial behaviors for predatory animals. The aim was to shed light on the basic mechanisms from which territoriality emerges, and to understand what are the consequences of these behaviors at the scale of the environment and the total populations of predators and prey. From a mathematical viewpoint, in \cite{BerestyckiZilio_PI} we have shown existence and uniqueness results of the parabolic version of \eqref{eqn model k same}, explored the asymptotic limit when the competition $\beta$ is very large, and we have obtained results about the existence of non-constant stationary solutions in the special case of $N = 2$ number of groups of predators. Then, in \cite{BerestyckiZilio_PII} we have shown that the solutions of \eqref{eqn model k same} are uniformly bounded in H\"older norm, independently of the value of $\beta$ and $N$ (see Theorem \ref{prp asymptotic k} below). This has allowed us to strengthen our conclusion about the asymptotic limit of large competition that we derived in \cite{BerestyckiZilio_PI}.

System \eqref{eqn model k same} adapts to the present context the classical model of Lotka and Volterra for predators and prey \cite{Volterra}. In this model, the interaction of two population is represented by the product of their densities. For Lotka, the motivation for this term came from the law of mass action in chemistry. Volterra's approach was to derive the interaction from the probability of encounter between individuals of the two populations. This probability, in turn, can be shown to be approximately proportional to the product of the densities. We keep this interaction between prey and predators in the terms $k u w_i$ in \eqref{eqn model k same}. Following this idea, we introduce the new terms $\beta w_i w_j$ in \eqref{eqn model k same} to represent the hostile interaction between different groups of predators.

In the mathematical literature, competition systems have been considered rather recently. To our knowledge, the study of strong competition can be traced back to the pioneering work of Dancer and Du \cite{DancerDu}. There, the authors considered a system of only two competing densities, without the distinction between prey and predators. They establish compactness results for the set of solutions, independently of the strength of the competition (in our model, this corresponds to the parameter $\beta$). Then, by means of a topological argument, they showed that their compactness results lead to existence and multiplicity results for the original model with strong but finite competition. This idea was developped further, and more precise results about the asymptotic behavior of solutions as the competition diverges were obtained by Conti, Terracini and Verzini in \cite{ContiTerraciniVerzini_AdvMat_2005}. More recently Soave and Zilio \cite{SoaveZilio_ARMA} were able to cover the case of uniform Lipschitz estimates. It is known that solutions cannot be uniformly continuous first derivatives, making these estimates optimal in the class of $C^{k,\alpha}$ spaces.

On the other hand, the study of the limit problem, that is, the one obtained in the limit of infinite competition, was first considered by Conti, Terracini and Verzini in dimension two \cite{CTV_indiana}. There, the authors showed that limit configurations are made of segregated densities (that is, densities whose supports have disjoint interior), and the interfaces between different densities have a rigid and regular structure. These questions were later addressed in any dimension, including the parabolic case, by Caffarelli, Karakhanyan and Lin \cite{CaffKarLin, CaffLin}. The original proofs in dimension two relied on the geometry of the plane. The articles of Caffarelli, Karakhanyan and Lin \cite{CaffKarLin, CaffLin} introduce some deep original ideas. In particular, in their proofs we find delicate applications of the classical Alt-Caffarelli-Friedman monotonicity formula \cite{ACF}, the Caffarelli monotonicity formula \cite{CaffSalsa} and the \emph{improvement of flatness} technique developped by Caffarelli. Later, Dancer, Wang and Zhang clarified even further the segregation phenomenon (that is, the limit of strong competition) and gave also an account on the speed of convergence of the densities \cite{DaWaZa_Dynamics}.

More recent developments include the study of non-local diffusion operators by Verzini and Zilio \cite{VerZilio} and non-local competition in a work of Caffarelli, Patrizi and Quitalo \cite{CaffPatrQuit}. The interest in these problems is twofold. On the one hand, many models are non-local in the original formulations. On the other hand, in the non-local framework many techniques of the standard local formulations are not available any longer. This has lead to the development of new approaches to these problems.

In this paper, we are concerned with the set of solutions of \eqref{eqn model k same} in two extreme cases: small competition ($\beta$ small and $N$ arbitrary) or large number of competing groups of predators ($N$ large and $\beta$ arbitrary). In a sense, this is a dual scenario with respect to $N$ bounded and $\beta$ very large, which was the main framework considered in \cite{BerestyckiZilio_PI}. We establish here that solutions are necessarily constant and unstable if $\beta$ is small or $N$ is large:
\begin{theorem}\label{main thm}
There exist $\bar \beta > 0$ and $\bar N \geq 1$ such that if $\beta \in [0, \bar \beta)$ or $N \geq \bar N$, then the set of solution of \eqref{eqn model k same} consists only of the constant solution
\[
  w_1 = \dots w_N = \frac{\lambda k - \mu \omega}{\mu \beta (N-1) + N k^2}, \quad u = \frac{\lambda \beta (N-1) + \omega k N }{\mu \beta (N-1) + k^2 N}.
\]
These solutions are (strongly) unstable if and only if $N \neq 1$ and $\beta > 0$.
\end{theorem}

The previous result bears consequence for the ecological interpretation of the model. Indeed, let us consider the total population of predators (the densities $w_i$) that reside in the domain $\Omega$. Then, if $\beta$ is small or $N$ is large enough, we find that its value is given by
\[
	\mathcal{W}_N = \int_{\Omega} \sum_{i=1}^{N} w_i  = \frac{(\lambda k - \mu \omega)N}{\mu \beta (N-1) + N k^2} |\Omega|.
\]
If the population has only one pack, that is, with no inter-specific competition, we have
\[
	\mathcal{W}_1 = \frac{(\lambda k - \mu \omega)}{k^2} |\Omega|.
\]
Hence, for $\beta$  small or $N$ large, the division into packs without formation of territories (constant solution) is always less advantageous in term of the total size of the population.

Let us briefly describe the strategy of the proof of Theorem \ref{main thm}. We first show some asymptotic results (Propositions \ref{prp beta to 0} and \ref{prp Nbeta}) which state that if $\beta$ is small or $N$ is large, then the solutions are uniformly close to constant solutions. Then we show that close-to-constant solutions are necessarily constant. The main difficulty in the proof of these results is that the number of component $N$ (and in some cases also the competition strength $\beta$) may be unbounded, and all the estimates need to be uniform in the parameters.

Many questions regarding the solutions of this system remain open that we think are worthy of further investigations.

\noindent\textbf{Open problem 1.} The thresholds $\bar \beta$ and $\bar N$ of Theorem \ref{main thm} are not explicit. It would be very relevant for modeling issues to have some estimates on these two quantities.

\noindent\textbf{Open problem 2.} Under which conditions do there exist non-constant solutions, outside of the region of Theorem \ref{main thm}? We answered this question \cite{BerestyckiZilio_PI} in dimension one and in higher dimension in rectangular domains. Also we have partial results in this direciton for the case of general domains in higher dimension for the case two groups of predators. More general existence results are known for similar systems that only involve competing preys (or predators with a given constant resource)  \cite{DancerDu, DancerDu_ExI, DancerDu_ExII}. But it is open in general for system \eqref{eqn model k same}.

\noindent\textbf{Open problem 3.} Solutions of \eqref{eqn model k same} are the stationary solutions of the parabolic system
  \[
  	\begin{cases}
		\partial_t w_i - d \Delta w_i  = \left(- \omega + k u - \beta \sum_{j \neq i} w_j\right) w_i &\text{in $\Omega \times (0,T)$}\\
		\partial_t u - D \Delta u = \left(\lambda - \mu u - k \sum_{i=1}^N w_i \right)u &\text{in $\Omega\times (0,T)$}\\
		\partial_\nu w_i = \partial_\nu u = 0 &\text{on $\partial \Omega \times (0,T)$}\\
		w_i(x,0) = w_{i,0}(x), u(x,0) = u_0(x) &\text{on $\Omega \times \{0\}$}.
	\end{cases}
	\]
	Under which conditions do solutions of the parabolic system converge to stationary solutions for large time $T$? An answer to this question together with Theorem \ref{main thm} could imply that the set $N \leq \bar N$ attracts the dynamics for large time. 
	
\noindent\textbf{Open problem 4.} We use the structure of the system in several steps of the proof of Theorem \ref{main thm}. In particular, our method relies on the fact that the coefficients in the equation of the predator densities (that is, $d$, $\omega$ and $k$), are independent of the density. It is an interesting an open problem to know whether or when the same type of classification results hold for the more general system
\begin{equation}\label{model gen}
	\begin{cases}
		- d_i \Delta w_i  = \left(- \omega_i + k_i u - \beta \sum_{j \neq i} a_{ij} w_j\right) w_i &\text{in $\Omega$}\\
		- D \Delta u = \left(\lambda - \mu u - k_i \sum_{i=1}^N w_i \right)u &\text{in $\Omega$}\\
		\partial_\nu w_i = \partial_\nu u = 0 &\text{on $\partial \Omega$}.
	\end{cases}
\end{equation}
Assuming, for instance, that $d_i$, $\omega_i$, $k_i$ and $a_{ij}$ are only close to values that are independent of $i$ and $j$. is it true that the only solutions of $\beta$ small or $N$ large are constant?

\begin{figure}
\begin{tikzpicture}
  \begin{axis}[
    axis lines = left,
    axis on top=true,
    axis equal image,
    xmin = 0, xmax = 10,
    ymin = 0, ymax = 4,
    clip=false,
    ticks=none,
    xlabel = $\beta$,
    ylabel = $N$,
    ylabel style={rotate=-90}
  ]
    
    \addplot[name path = A, draw = none, fill = orange!25] coordinates {(0,0) (1,0) (1,2) (10,2) (10,4) (0,4)} \closedcycle;
     
    \addplot +[mark=none, black, dashed] coordinates {(1, 0) (1, 2) (10, 2)};

    \node[anchor = west] at (axis cs: 0, 2) {$\bar N$};
    \node[anchor = north] at (axis cs: 1,0) {$\bar \beta$};
    
    \node[] at (axis cs: 5, 3) {only constant solutions};  
    %\node[] at (axis cs: 5, 1) {non-constant solutions?};  
    
    \end{axis}
\end{tikzpicture}
\caption{Pictorial description of Theorem \ref{main thm}.}
\end{figure}
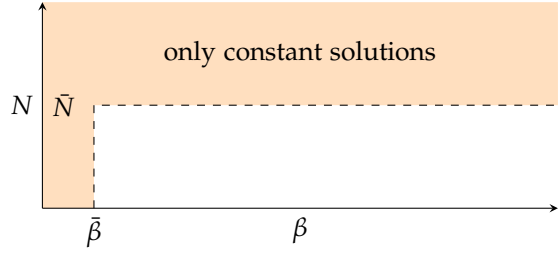

\bigskip
\noindent \textbf{Acknowledgements:}  This work has been supported by the ERC Advanced Grant 2013 n. 321186 ``ReaDi -- Reaction-Diffusion Equations, Propagation and Modelling'' held by Henri Berestycki, and by the French National Research Agency (ANR), within  project NONLOCAL ANR-14-CE25-0013. 
Part of this work was completed while the first author was visiting the Hong Kong University of Science and Technology Jockey Club Institute of Advanced Study whose support he gratefully acknowledges.

\section{Preliminary results}
We start by stating here some already known results that will be useful in the following. First, we recall that positive solutions of \eqref{eqn model k same} are uniformly bounded independently of $\beta \geq 0$ and $N$. More precisely, in \cite{BerestyckiZilio_PII} we have shown the following estimate.

\begin{theorem}\label{prp asymptotic k}
Let $\Omega \subset \R^n$ be a smooth domain. Let $\beta \geq 0$ and $N \in \N$. We consider a non negative (bounded) solution $\mathbf{v} = (w_1, \dots, w_N, u) = (\mathbf{w},u)$ of the system \eqref{eqn model k same}. Then  all components of $\mathbf{v}$ are uniformly bounded in $L^\infty(\Omega)$ with respect to $\beta > 0$ and $N \in \N$, and there exists $C = C(\Omega) >0$ (that, in particular, is independent of $\beta$ and $N$) such that
\[
  0 \leq u \leq \frac{\lambda}{\mu}, \qquad \text{and} \qquad 0 \leq \sum_{i=1}^N w_i \leq C.
\]
Moreover, for any $\alpha \in (0,1)$ there exists $C_\alpha = C(\alpha, \Omega)$ (again independent of $\beta$ and $N$) such that
\[
	\| u \|_{\C^{2,\alpha}(\overline \Omega)} \leq C_\alpha
\]
and
\[
  \max_{i \in \{1, \dots, N\}} \| w_i \|_{C^{0,\alpha}(\overline{\Omega})} + \left\| \sum_{1=1}^N w_i \right\|_{C^{0,\alpha}(\overline{\Omega})} \leq C_\alpha \left\| \sum_{1=1}^N w_i \right\|_{L^{\infty}(\overline{\Omega})}.
\]
\end{theorem}
Actually, the result in \cite{BerestyckiZilio_PII} is more general that the one stated here. For instance, in \cite{BerestyckiZilio_PII} we did not assume that the coefficients in the equations are independent of the index $i$.

In \cite{BerestyckiZilio_PII}, we have also conducted a first asymptotic analysis of the solutions for $\beta$ large, which we recall.
\begin{theorem}\label{thm hat}
There exist $\hat \beta > 0$ and $\hat N \in \N$ sufficiently large, such that if $\beta > \hat \beta$ and $\mathbf{v}_\beta = (\mathbf{w}_{\beta}, u_{\beta})$ is a solution of \eqref{eqn model k same} then
\begin{itemize}
	\item either at most $\hat N$ components of $\mathbf{w}_\beta$ are strictly positive and the others are zero (that is $N \leq \hat N$);
	\item or, in case $N > \hat N$, the solution is such that
\[
	 \max_{i=1,\dots, N} \|w_{i,\beta}\|_{\C^{0,\alpha}(\Omega)} + \|u_\beta- \lambda/\mu\|_{\C^{2,\alpha}(\Omega)} = o_\beta(1)
\]
for every $\alpha \in (0,1)$.
\end{itemize}
\end{theorem}

% \begin{lemma}\label{lem no order}
% Assume $\beta > 0$. If there exists $i \neq j$ such that $w_i \geq w_j$ in $\Omega$, then $w_i \equiv w_j$ or $w_j \equiv 0$.
% \end{lemma}

\section{Stability properties of constant solutions}

We start now with the core arguments of the paper. First of all, we analyze the constant solutions of \eqref{eqn model k same} and investigate their stability. The results contained in this section follow rather straightforward computations, but are fundamental in our argument. Thus we detail them. Specifically, we show that \eqref{eqn model k same} has only constant solutions if $\beta = 0$ (this is a generalization of \cite[Lemma 3.2]{BerestyckiZilio_PI}). These are unstable if $N = 0$ and stable if $N\geq 1$. If $\beta > 0$, constant solutions are uniquely determined by the number of their non-zero components. Furthermore, in this case they are strongly linearly stable if $N = 1$ and (strongly) linearly unstable otherwise. Thus, this section contains a generalization of \cite[Lemma 3.3]{BerestyckiZilio_PI}, which was stated in the case $N \leq 2$.

We begin this section by recalling a result of Mimura \cite[Theorem 1]{Mimura}. It concerns the solutions of the classical predator-prey model with Neumann boundary conditions, and it states that these are necessarily  constant. Here we give a short and more precise proof of this useful result.

\begin{lemma}\label{lem mimura}
The non-negative (bounded) solutions of system 
\[
	\begin{cases}
		-D \Delta H = (-\omega + k u - \beta H) H &\text{in $\Omega$}\\
		-d \Delta u = (\lambda - \mu u - k H) u &\text{in $\Omega$}\\
		\partial_\nu H = \partial_\nu u = 0 &\text{on $\partial \Omega$}
	\end{cases}
\]
are all constant. In particular, they are
\[
	(H,u) = (0,0),\; \left(0, \frac{\lambda}{\mu}\right) \; \text{or} \; \left(\frac{\lambda k - \mu \omega}{k^2 + \mu \beta},   \frac{\omega k + \lambda \beta}{k^2 + \mu \beta}\right).
\]
The result holds true even if $\beta = 0$ or $\mu = 0$.
\end{lemma}
\begin{proof}
We can apply the comparison principle to the equations in $H$ and $u$ separately. It follows that either the corresponding component is $0$ or it is strictly positive. In case one component is $0$, we readily deduce that the other one must be constant. The only two possibilities are $(H,u) = (0,0)$ or $(H,u) = (0, \lambda/\mu)$.

Thus we only need to consider the case of strictly positive solutions. Let
\[
	\mathfrak{h} = \frac{\lambda k - \mu \omega}{k^2 + \mu \beta}, \; \mathfrak{u} = \frac{\omega k + \lambda \beta}{k^2 + \mu \beta}.
\]
Observe that $(\mathfrak{h}, \mathfrak{u})$ is a solution of the system. In particular, we have the following identities
\[
 \mathfrak{h} = \frac{\lambda - \mu \mathfrak{u}}{k}, \; \mathfrak{u} = \frac{\omega}{k} + \frac{\beta}{k} \mathfrak{h}.
\]
We use these identities and the equations in the system to compute the integral
\[
	I = \int_\Omega \left(1-\frac{\mathfrak{h}}{H}\right) D \Delta H + \left(1-\frac{\mathfrak{u}}{u}\right) d \Delta u.
\]
We find
\[
	\begin{split}
		I &= - \int_\Omega \left(1-\frac{\mathfrak{h}}{H}\right) \left(-\omega + k u - \beta H\right)H + \left(1-\frac{\mathfrak{u}}{u}\right) \left(\lambda - \mu u - k H\right) u\\
		&= \int_\Omega \beta H^2 - 2 \beta \mathfrak{h} H + \beta \mathfrak{h}^2 + \mu u^2 - 2 \mu \mathfrak{u} u + \mu\mathfrak{u}^2\\
		&= \int_\Omega  \beta \left[H - \mathfrak{h}\right]^2 + \mu \left[u - \mathfrak{u}\right]^2 \geq 0.
	\end{split}
\]
On the other hand
\[
	I = - \mathfrak{h} \int_\Omega D \frac{|\nabla H|^2}{H^2} - \mathfrak{u} \int_\Omega d \frac{|\nabla u|^2}{u^2} \leq 0.
\]
Combining the two estimates, it follows that both $u$ and $H$ must be positive and constant. From the system we find that $H = \mathfrak{h}$ and $u = \mathfrak{u}$. This concludes the proof. 
\end{proof}

We can use the previous result in order to completely classify the solutions of \eqref{eqn model k same} in the case $\beta = 0$.

\begin{lemma}\label{lem class beta 0}
Assume $\beta = 0$. Then all solutions $\mathbf{v}=(\mathbf{w},u)$ of \eqref{eqn model k same} are constant. If $\mathbf{w} = 0$, then the corresponding solution is unstable. On the other hand, if at least one component of $\mathbf{w}$ is positive, then $\mathbf{w}$ can be any vector of non negative components such that
\[
	\sum_{i=1}^{N} w_i = \frac{\lambda k - \mu \omega}{k^2}, \qquad \text{and} \qquad u = \frac{\omega}{k}.
\]
These solutions are weakly linearly stable.
\end{lemma}
%Observe that in this case, $N$ can also be taken $=+\infty$.
\begin{proof}
We consider the function $H = \sum_{i=1}^N w_i$. Summing all the equations in $w_i$, we find the reduced system
\[
	\begin{cases}
		- d \Delta H  = \left(- \omega + k u \right) H&\text{in $\Omega$}\\
		- D \Delta u = \left(\lambda - \mu u - k H \right)u &\text{in $\Omega$}\\
		\partial_\nu H= \partial_\nu u = 0 &\text{on $\partial \Omega$}.
	\end{cases}
\]
As by Lemma \ref{lem mimura}, $(H,u)$ must be constant, and it has to be one of the following three solutions
\[
	(H,u) = (0,0), \; \left(0, \frac{\lambda}{\mu}\right) \; \text{or} \; \left(\frac{\lambda k - \mu \omega}{k^2},  \frac{\omega}{k}\right).
\]
We can easily verify that the first two solutions are strongly linearly unstable (this follows also from \cite[Lemma 3.3]{BerestyckiZilio_PI}). In the third case, we find that each component $w_i$ of $\mathbf{w}$ is solution to
\[
	\begin{cases}
		- d \Delta w_i  = \left(- \omega + k \frac{\omega}{k} \right) w_i = 0&\text{in $\Omega$}\\
		\partial_\nu w_i= 0 &\text{on $\partial \Omega$}.
	\end{cases}	
\]
Thus, $\mathbf{w}$ is made of positive constant functions whose sum is $(\lambda k - \mu \omega) / k^2$. We observe that in this case, the solutions form an open and non empty simplex. By reasoning as in \cite[Lemma 3.3]{BerestyckiZilio_PI} we deduce that they are weakly linearly stable.
\end{proof}

We now consider the case $\beta > 0$.

\begin{lemma}\label{lem stab const sol}
Let $\beta >0$. For any $N \in \N_*$ there exists a unique positive constant solution $\mathbf{v} = (\mathbf{w}, u)$ of \eqref{eqn model k same}, given by
\[
  w = w_1 = \dots w_N = \frac{\lambda k - \mu \omega}{\mu \beta (N-1) + N k^2}, \quad u = \frac{\lambda \beta (N-1) + \omega k N }{\mu \beta (N-1) + N k^2 }.
\]
This solution is linearly stable if $N = 1$, and strongly unstable if $N \geq 2$. In the latter case, the linearized system at $\mathbf{v}$ writes as follows
\[
  \begin{cases}
    - \Delta \begin{pmatrix}d \mathbf{W} \\ D U \end{pmatrix} = A_\beta \begin{pmatrix} \mathbf{W} \\ U \end{pmatrix} &\text{in $\Omega$}\\
    \partial_\nu \boldsymbol{W} = \partial_\nu U = 0  &\text{on $\partial \Omega$},
  \end{cases}
\]
where the matrix $A_\beta$ is
\[
  A_\beta = \begin{pmatrix}
    0 & -\beta w & -\beta w & \dots & - \beta w & k w\\
    -\beta w & 0 & -\beta w & \dots & - \beta w & k w\\
    -\beta w & -\beta w & 0 & \dots & - \beta w & k w\\
    \vdots & \vdots & \vdots & \ddots & \vdots & \vdots\\
    -\beta w & -\beta w & -\beta w & \dots & 0 & k w\\
    - k u & - k u & - k u & \dots & - k u & - \mu u
  \end{pmatrix}.
\]
The eigenvalues of $A_\beta$ are $\{ \beta w, \Lambda_1, \Lambda_2\}$, where
\begin{itemize}
  \item $\beta w$ has multiplicty $N-1$, and its eigenspace is given by the vectors $\mathbf{V} = (\mathbf{W},U) \in \R^{N+1}$ such that
  \[ 
    \sum_{i=1}^N W_i = 0, \quad U=0;
  \]
  \item $\Lambda_1$ and $\Lambda_2$ have strictly negative real part.
\end{itemize}
\end{lemma}
Observe that, since $d \neq D$, the knowledge of the spectrum of $A_\beta$ in general does relate to the stability/instability of the solutions. However, in this case the unstable directions are given by constant functions, thus we can infer the stability properties of the solutions from the spectrum of $A_\beta$. 
\begin{proof}
One can easily verify that the function $\mathbf{v}$ in the statement is a solution to \eqref{eqn model k same}. The uniqueness of the solution follows from the fact that the matrix of the corresponding linear system is invertible. A direct computation yields the following formula for the characteristic polynomials of $A_\beta$:
\[
	\det(A_\beta - \gamma \mathrm{Id}) =  (\beta w - \gamma)^{N-1} \left[ \gamma^2 + \gamma (\mu u + (N-1)\beta w) + ((N-1)\beta \mu + Nk^2)uw \right].
\]
Therefore, $A_\beta$ has eigenvalues $\beta w$, whose multiplicity is $N-1$, and the remaining eigenvalues are complex conjugate and have strictly negative real part. The eigenspace of $\beta w$ is spanned by the vectors $\mathbf{V} = (\mathbf{W}, U)$ such that
\[
  \begin{cases}
    \beta \sum_{i=1}^N W_i = k U\\
    k \sum_{i=1}^N W_i = - \mu U
  \end{cases} \implies 
  \begin{cases}
    \sum_{i=1}^N W_i = 0\\
    U = 0.
  \end{cases}
\]
We see that at least one component of $\mathbf{W}$ is negative. 
\end{proof}

\section{Asymptotic results of positive solutions for \texorpdfstring{$\beta$}{beta} small}
We now turn our attention to the study of general positive solutions and analyze the behavior of the solutions of system \eqref{eqn model k same} as a function of the parameter $\beta$ and $N$. To start with, we first consider the case of small $\beta > 0$. Our first aim is to show that every solution of \eqref{eqn model k same} is close to the constant solutions in a strong sense, which we describe in the following proposition. This is a generalization of \cite[Proposition 3.14]{BerestyckiZilio_PI}. However, note that since the number of components $N$ is not a priori fixed, but is here a free parameter, this is a quite delicate extension and requires new ingredients in the proof.

\begin{proposition}\label{prp beta to 0}
For any $\eps > 0$ there exists $\beta_\eps > 0$ such that for any $\mathbf{v}= (\mathbf{w}, u)$ solution of  \eqref{eqn model k same} with $\beta \in [0,\beta_\eps]$ and any $N \geq 1$, the following estimates hold
\[
	\begin{split}
	\left\| u - \frac{\omega}{k} \right\|_{C^{2,\alpha}(\Omega)} &\leq \eps\\
	\max_{i=1,\dots,N} \left\| w_{i} - \frac{\lambda k - \mu \omega}{N k^2} \right\|_{C^{2,\alpha}(\Omega)} &\leq \frac{\eps}{N}.
	\end{split}
\]
Letting $H = \sum_{i=1}^{N} w_{i}$, this entails that 
\[
  \left\| H - \frac{\lambda k - \mu \omega}{k^2} \right\|_{C^{2,\alpha}(\Omega)} \leq \eps.
\]
\end{proposition}
Before proceeding, we provide a technical lemma that will be of use later on. The proof is straightforward and is left to the reader.
\begin{lemma}\label{tech lemma}
Let $\mathbf{w} = (w_1, \dots, w_N) \subset C^{0,\alpha}(\Omega)$ be a vector of non negative functions. Then
\[
  \left\|\sum_{i = 1}^N w_i^2\right\|_{C^{0,\alpha}} \leq \left(\max_{i=1,\dots,N}\|w_i\|_{L^\infty} + 2 \max_{i=1,\dots,N}|w_i|_{C^{0,\alpha}} \right) \left\|\sum_{i = 1}^N w_i\right\|_{L^\infty}.
\]
\end{lemma}

Finally we recall that, if $\beta > 0$, the solutions satisfy the following rigidity property with respect to ordering.
\begin{lemma}\label{lem no order}
Assume $\beta > 0$. If there exists $i \neq j$ such that $w_i \geq w_j$ in $\Omega$, then $w_i \equiv w_j$ or $w_j \equiv 0$.
\end{lemma}
This simple lemma is very useful in many of our arguments. It will allows us to show that the components $\mathbf{w}$ of the solutions have all similar behaviors. Extending this property to model the more general framework of system \eqref{model gen} would essentially allows us to establish Theorem \ref{main thm} in the greater generality. This lemma extends a result of \cite[Proposition 3.14]{BerestyckiZilio_PI} stated and proved in the case $N=2$ to the case of $N$ components. For completeness, we give here a short proof.

\begin{proof}
We consider a solution of \eqref{eqn model k same} and we assume that there exist $i \neq j$ such that $w_i \geq w_j \geq 0$ and $w_i \not \equiv 0$. We look at the equations satisfied by $w_i$ and $w_j$. By letting
\[
  g = \beta w_i w_j \geq 0
\]
We have
\begin{equation}
	\begin{cases}
		- d \Delta w_i  = \left(- \omega + k u - \beta \sum_{h \neq i,j} w_h\right) w_i - g &\text{in $\Omega$}\\
		- d \Delta w_j  = \left(- \omega + k u - \beta \sum_{h \neq i,j} w_h\right) w_j - g &\text{in $\Omega$}\\
		\partial_\nu w_i = \partial_\nu w_j = 0 &\text{on $\partial \Omega$}.
	\end{cases}
\end{equation}
Thus $w_i$ and $w_j$ are two solutions of the same linear elliptic equation. An integration by parts and Green's formula show that:
\[
  \int_{\Omega} g (w_i - w_j) = 0.
\]
Hence, $w_j \equiv 0$ (in which case $g \equiv 0$) or $w_i \equiv w_j$.
\end{proof}

\begin{proof}[Proof of Proposition \ref{prp beta to 0}]
Let $\mathbf{v}_n=(\mathbf{w}_n,u_n)$ be any sequence of solutions of \eqref{eqn model k same} defined for $\beta_n > 0$ and $\beta_n \to 0$. Before deriving the behavior of each component of the vector $\mathbf{w}_n$, we first start with an estimate of the sum $H_n$. We first show that
\[
  \lim_{n\to+\infty}\left\| H_n - \frac{\lambda k - \mu \omega}{k^2} \right\|_{C^{2,\alpha}} + \left\| u_n - \frac{\omega}{k} \right\|_{C^{2,\alpha}}= 0.
\]

By Theorem \ref{prp asymptotic k}, we already know that the sequence $(H_n,u_n)$ is uniformly bounded in $C^{0,\alpha}\times C^{2,\alpha}(\overline{\Omega})$. Thus, up to a subsequence, $(H_n,u_n)$ converges to a limit profile $(H,u)$. We now derive a limit system for $(H,u)$ and show that necessarily $(H,u) = ((\lambda k - \mu \omega)/ k^2, \omega/k)$. The identification of a single possible limit then implies that the whole sequence converges to it.

First, we have that $(H_n,u_n)$ are solutions of 
\[
	\begin{cases}
		- d \Delta H_n  = \left(- \omega + k u_n - \beta_n H_n \right) H_n  + \beta_n \sum_{i=1}^{N_n} w_{i,n}^2 &\text{in $\Omega$}\\
		- D \Delta u_n = \left(\lambda - \mu u_n - k H_n \right) u_n &\text{in $\Omega$}\\
		\partial_\nu H_n = \partial_\nu u_n = 0 &\text{on $\partial \Omega$}.
	\end{cases}
\]
The components of $\mathbf{w}_n$ being non-negative, we know that
\[
  0 \leq \sum_{i=1}^{N_n} w_{i,n}^2 \leq \left( \sum_{i=1}^{N_n} w_{i,n} \right)^2 = H_n^2.
\]
Since $\beta_n \to 0$ and $H_n$ is uniformly bounded, $\Delta H_n$ is uniformly bounded in $C^{0,\alpha}(\Omega)$ for any $\alpha \in (0,1)$ (see Lemma \ref{tech lemma}). By standard elliptic regularity, $H_n$ is uniformly bounded in $C^{2,\alpha}(\overline{\Omega})$ for any $\alpha \in (0,1)$. Exploiting this information, we also obtain that $u_n$ is uniformly bounded in $C^{2,\alpha}(\overline{\Omega})$ for any $\alpha \in (0,1)$. Thus, up to striking out a subsequence, we get that $(H_n, u_n) \to (H,u)$ in $C^{2,\alpha}(\overline{\Omega})$ for any $\alpha \in (0,1)$. By passing to the limit in the equation, we see that $(H,u)$ is a solution of
\[
	\begin{cases}
		- d \Delta H  = \left(- \omega + k u \right) H &\text{in $\Omega$}\\
		- D \Delta u = \left(\lambda - \mu u - k H \right) u &\text{in $\Omega$}\\
		\partial_\nu H = \partial_\nu u = 0 &\text{on $\partial \Omega$}.
	\end{cases}
\]
Hence, by Lemma \ref{lem mimura}, we conclude that $(H,u)$ must be constant. Thus we have three possibilities
\[
	(H,u) = (0,0), \; \text{or} \; \left(0, \frac{\lambda}{\mu}\right) \; \text{or} \; \left(\frac{\lambda k - \mu \omega}{k^2},   \frac{\omega}{k}\right).
\]
Our goal is to prove that only the last one can occur.

We first show that, necessarily, $u > 0$. Indeed, assume by contradiction that the component $u_n$ converges (uniformly) to $0$. Then, there exists $n_0 \in \N$ such that $u_{n_0} < \omega / k$. But then the maximum principle, when applied to the equation for $H_{n_0}$, implies that necessarily $H_{n_0} \equiv 0$, that is $\mathbf{w}_{n_0} \equiv 0$, a contradiction with the assumption $N_n \geq 1$ for all $n\in\N$.

We now show that $u <  \lambda / \mu$. Reasoning again by contradiction, we assume that $u \to \lambda / \mu$ (uniformly). We consider the normalized function
\[
  \bar H_n = \frac{H_n}{\|H_n\|_{L^\infty}}.
\]
This new sequence of functions verifies
\[
	\begin{cases}
		- d \Delta \bar H_n  = \left(- \omega + k u_n - \beta_n H_n \right) \bar H_n  + \beta_n \left(\sum_{i=1}^{N_n} w_{i,n}^2 \left/ \|H_n\|_{L^\infty} \right.\right) &\text{in $\Omega$}\\
		\partial_\nu \bar H_n = 0 &\text{on $\partial \Omega$}.
	\end{cases}
\]
Once more by the uniform estimates in Theorem \ref{prp asymptotic k}, we find that the sequence $\{\bar H_n\}_n$ is uniformly bounded in $C^{0,\alpha}(\overline{\Omega})$ and, by the previous equation, we also derive that $\{ \bar H_n\}_n$ is uniformly bounded in $C^{2,\alpha}(\Omega)$ for any $\alpha \in (0,1)$. As a result, up to a subsequence, $\{\bar H_n\}_n$ converges to a non-negative function $\bar H \in C^{2,\alpha}(\Omega)$ solution of
\[
	\begin{cases}
		- d \Delta \bar H  = \left( \frac{\lambda}{\mu} k - \omega \right) \bar H  &\text{in $\Omega$}\\
		\partial_\nu \bar H = 0 &\text{on $\partial \Omega$}
	\end{cases}
\]
with $\|\bar H\|_{L^\infty} = 1$. Owing to the assumption that $\lambda k > \omega \mu$, we must have $\bar H \equiv 0$, a contradiction.

Thus $u = \omega / k$ and, necessarily, $H = (\lambda k - \mu \omega)/k^2$. Therefore, we find that the whole sequence $(H_n, u_n)$ converges to $( (\lambda k - \mu \omega)/k^2, \omega /k)$. 

To conclude the proof, we only need to show that each component of $\mathbf{w}_n$ converges to the same (scaled) constant. First, by letting
\[
	\bar w_{i,n} = \frac{w_{i,n}}{\|w_{i,n}\|_{L^\infty}}
\]
we have that
\[
	\begin{cases}
		- d \Delta \bar w_{i,n}  = \left( -\omega + k u_n - \beta_n H_n + \beta_n w_{i,n} \right) \bar w_{i,n} &\text{in $\Omega$}\\
		\partial_\nu \bar w_{i,n} = 0 &\text{on $\partial \Omega$}.
	\end{cases}
\]
We recall that $H_n$ is uniformly bounded in $C^{2,\alpha}(\overline{\Omega})$ for any $\alpha \in (0,1)$, and $w_{i,n}$ is uniformly bounded in $C^{0,\alpha}(\overline{\Omega})$. Moreover by assumption $\beta_n \to 0$. By the same reasoning as before, up to a subsequence, the sequence $\bar w_{i,n}$ converges to a non negative function $\bar w_i \in C^{2,\alpha}(\Omega)$ such that $\|\bar w_i\|_{L^\infty} = 1$ and
\[
	\begin{cases}
		- d \Delta \bar w_{i}  = \left( -\omega + k u \right) \bar w_{i} = 0 &\text{in $\Omega$}\\
		\partial_\nu \bar w_{i} = 0 &\text{on $\partial \Omega$},
	\end{cases}
\]
therefore $\bar w_{i} \equiv 1$. Now, assume that there exists $\eps > 0$ and $\bar n$ large such that
\[
	\|w_{i_n,n}\|_{L^\infty} < (1-\eps) \|w_{j_n,n}\|_{L^\infty}
\]
for any $n \geq \bar n$ and indexes $i_n, j_n \in \{1, \dots, N_n\}$. Then there exists $n$ large enough such that $w_{i_n,n} < w_{j_n,n}$. By Lemma \ref{lem no order}, since $\beta_n > 0$, this yields $w_{i_n,n} \equiv 0$, a contradiction. Thus, we have that
\begin{equation}\label{eqn two limits}
  \lim_{n\to+\infty} \sup_{i,j \in \{1, \dots, N_n\}} \frac{\| w_{i,n}\|_{L^\infty}}{\| w_{j,n}\|_{L^\infty}} = 1 \qquad \text{and} \qquad  \lim_{n \to +\infty} \sup_{i,j \in \{1, \dots, N_n\}} \left\| \frac{w_{i,n}}{\| w_{j,n}\|_{L^\infty}}-1 \right\|_{C^{2,\alpha}} =  0.
\end{equation}
From the first limit it follows that there exists a sequence $\eps_n \to 0$ such that
\[
  (1-\eps_n)\sup_{j = 1, \dots, N_n} \| w_{j,n}\|_{L^\infty} \leq w_{i,n}(x) \leq (1+\eps_n) \inf_{j = 1, \dots, N_n} \| w_{j,n}\|_{L^\infty}
\]
for all $n \in \N$, $i = 1, \dots, N_n$ and $x \in \overline{\Omega}$. Summing up in $i$ we find
\[
  (1-\eps_n)N_n \sup_{j = 1, \dots, N_n} \| w_{j,n}\|_{L^\infty} \leq H_n \leq (1+\eps_n) N_n \inf_{j = 1, \dots, N_n} \| w_{j,n}\|_{L^\infty}.
\]
Combining this inequality the second limit in \eqref{eqn two limits}, we find
\[
	\lim_{n\to+\infty} N_n \sup_{i=1,\dots,N_n} \left\| w_{i,n} - \frac{\lambda k - \mu \omega}{N_n k^2} \right\|_{C^{2,\alpha}} = 0.
\]
This concludes the proof of the Proposition.
\end{proof}

\section{Asymptotic results of positive solutions for \texorpdfstring{$N$}{N} large}

In the preceding section, we have studied asymptotic results of positive solutions when $\beta$ is close to $0$ (independently of $N$). We now investigate what happens when $N$ is large (independently of $\beta>0$). We will show that the system has similar behaviors in both cases. We first prove in this section that if $N$ is large enough, independently of the value of $\beta$, then all solutions are close to the constant solutions of Lemma \ref{lem stab const sol}, in a sense to be specified. In the last section, we will show that solutions are actually constant. To prove this, an essential step is to prove that solutions are close to constants.

We now state the precise result in the following proposition which is the analogue of Proposition \ref{prp beta to 0} in the case of $N$ large. 
 
\begin{proposition}\label{prp Nbeta}
For any $\eps > 0$ there exists $N_\eps \in \N$ such that for any $\mathbf{v} = (\mathbf{w}, u)$ solution of \eqref{eqn model k same} with $N \geq N_\eps$ and any $\beta > 0$, we have
\[
  \left\| u - \frac{\lambda \beta (N-1) + \omega k N }{\mu \beta (N-1) + k^2 N} \right\|_{C^{2,\alpha}(\Omega)} \leq \eps
\]
and
\[
	\max_{i=1,\dots,N} \left\| w_{i} - \frac{\lambda k - \mu \omega}{\mu \beta (N-1) + N k^2} \right\|_{C^{2,\alpha}(\Omega)} \leq \frac{\eps}{N(1+\beta)}.
\]
In particular, letting $H = \sum_{i=1}^{N} w_{i}$, this implies
\[
  \left\| H - N \frac{\lambda k - \mu \omega}{\mu \beta (N-1) + N k^2} \right\|_{C^{2,\alpha}(\Omega)} \leq \frac{\eps}{1+\beta}.
\]
\end{proposition}
The proof of this proposition is rather involved, and it will be divided into several intermediate results. Our first aim is to show that if $N$ is large, then all components of any solution $\mathbf{w}$ are small in the uniform norm. Then, we will derive a uniform estimate on the sum of all the components of $\mathbf{w}$, showing in particular that it converges to zero if $\beta$ is large. Collecting all these intermediate steps, we will be able to conclude that the solutions converge to constant solutions for $N$ large, independently of $\beta$. As was the case in the previous section with respect to the dependence in $N$, the main difficulty here is that we want to obtain estimates that are uniform in $\beta$.

We start by showing that if $N$ becomes large, all the components of $\mathbf{w}$ converge to $0$.
\begin{lemma}\label{lem all to zero}
For any $\eps > 0$ there exists $N_\eps \in \N$ such that for any $\mathbf{v} = (\mathbf{w}, u)$ solution of \eqref{eqn model k same} with $N \geq N_\eps$ and $\beta > 0$, we have
\[
	\sup_{i=1, \dots, N} \|w_{i}\|_{L^\infty(\Omega)} \leq \eps.
\]
\end{lemma}
\begin{proof}
We argue by contradiction and assume that there exists a sequence of solutions $(\mathbf{w}_{n}, u_{n})$ of \eqref{eqn model k same}, a constant $\delta > 0$ and a sequence $\{ i_n : 1 \leq i_n  \leq N_n \}$ such that $N_n \to +\infty$ and
\[
  \|w_{i_n}\|_{L^\infty} \geq \delta.
\]
We consider the function $H_n = \sum_{i=1}^{N_n} w_{i,n}$. By the uniform estimates, we recall Theorem \ref{prp asymptotic k}, we know that, up to a subsequence, $H_n = \sum_{i=1}^{N_n} w_{i,n}$ converges in the $C^{0,\alpha}(\overline{\Omega})$ norm to some limit function $H$. Moreover, since $N_n \to +\infty$, by Theorem \ref{thm hat} we find that $\beta_n \leq \hat \beta$. Thus we can extract yet another subsequence and assume that $\beta_n \to \beta$, with $\beta \in [0, \hat \beta]$.

Again from the uniform bounds of Theorem \ref{prp asymptotic k} and the assumption that $N_n \to +\infty$, it follows that there exists at least a sequence $\{ j_n : 1 \leq j_n  \leq N_n \}$ such that $\|w_{j_n,n}\|_{L^\infty} \to 0$. Indeed, assume that this is not the case. Then there exists $\eta > 0$ such that $\|w_{i,n}\|_{L^\infty} \geq \eta$ for all $i$ and $n$. Exploiting the uniform $C^{0,\alpha}(\overline{\Omega})$ bounds, we can find a sequence $\{x_{i,n}\} \subset K$ and a radius $r > 0$ such that $w_{i,n}(x) \geq \eta / 2$ for all $x \in B_r(x_{i,n}) \cap \overline{\Omega}$. Thus, by Lemma \ref{lem cover} (see Appendix \ref{app cover}), for any $n \in \N$ we know that there exists a point $x_n \in \overline{\Omega}$ such that
\[
  \sum_{i=1}^{N_n} w_{i,n}(x_n) \geq C_r N_n \frac{\eta}{2}
\]
for a positive constant $C_r > 0$. We find a contradiction with uniform $L^\infty$ bound in Theorem \ref{prp asymptotic k}. 

Up to a relabelling, we assume that $i_n \equiv 1$ and $j_n \equiv 2$, so that $w_{1,n} \to w_1 \not \equiv 0$ and $w_{2,n} \to 0$ in the $C^{0,\alpha}(\overline{\Omega})$ topology for all $\alpha < 1$. Considering the equations satisfied by $w_{1,n}$, we have
\begin{equation}
	\begin{cases}
		- d \Delta w_{1,n}  = \left(- \omega + k u_n +\beta_n w_{1,n} - \beta_n H_{n}\right) w_{1,n} &\text{in $\Omega$}\\
		\partial_\nu w_{1,n} = 0 &\text{on $\partial \Omega$}.
	\end{cases}
\end{equation}
From this equation we infer that $w_{1,n}$ is bounded in $C^{2,\alpha}(\overline{\Omega})$ for all $\alpha \in (0,1)$ (recall that the sequence $\beta_n$ is bounded), and thus $w_{1,n} \to w_1$ in $C^{2,\alpha}(\overline{\Omega})$. We can pass to the limit in the equation and find
\begin{equation}
	\begin{cases}
		- d \Delta w_1  = \left(- \omega + k u - \beta w_1 - \beta H\right) w_1 &\text{in $\Omega$}\\
		\partial_\nu w_1 = 0 &\text{on $\partial \Omega$}.
	\end{cases}
\end{equation}
Since $w_1 \geq 0$ and by assumption $w_1 \not \equiv 0$ (indeed $\|w_{1}\|_{L^\infty} \geq \delta$), by the maximum principle we find that $w_1 > \eta$ in $\Omega$ for some positive constant $\eta$. As a consequence, for $n$ large enough we have that $w_{1,n} > \eta/2$ in $\Omega$. On the other hand, since $w_{2,n} \to 0$ uniformly in $\Omega$, for $n$ large enough we find $w_{2,n} < \eta/2$ in $\Omega$. But then there exists $\bar n > 0$ such that $w_{1,\bar n} > w_{2,\bar n}$, and by Lemma \ref{lem no order} this implies $w_{2,\bar n} \equiv 0$, a contradiction.
\end{proof}

Next we show that the sum of all components is bounded and decays for $\beta$ large.

\begin{lemma}\label{lem beta H bounded}
There exist $\bar N \in \N$ and $C > 0$ such that, for any $\mathbf{v} = (\mathbf{w},u)$ solution of \eqref{eqn model k same} with $N \geq \bar N$, we have
\[
	\|H\|_{L^\infty(\Omega)} \leq \frac{C}{1+\beta}.
\]
\end{lemma}
\begin{proof}
We argue by contradiction and assume that there exists a sequence $\mathbf{v}_n = (\mathbf{w}_n, u_n)$ of solutions such that $N_n \to +\infty$ and $\beta_n \|H_{n}\|_{L^\infty(\Omega)} \to + \infty$. We already know by Theorem \ref{prp asymptotic k} that the sequence $H_n$ is bounded in $C^{0,\alpha}(\overline{\Omega})$, thus we infer that $\beta_n \to +\infty$. We consider the following alternative.

\noindent \textbf{Case 1)} There exists $\eta > 0$ such that
\[
	\sup_{i=1,\dots,N_n} \|w_{i,n}\|_{L^\infty} \geq \eta \|H_n\|_{L^\infty}.
\]
Up to a relabeling, we can assume that $\|w_{1,n}\|_{L^\infty(\Omega)} \geq \eta \|H_n\|_{L^\infty(\Omega)}$. Let us introduce the scaled functions and sequence
\[
	\mathbf{\bar w}_{i,n} = \frac{\mathbf{w}_{i,n}}{ \|H_{n}\|_{L^\infty}}, \qquad \bar \beta_n := \beta_n  \| H_{n}\|_{L^\infty(\Omega)}.
\]
Thus $\bar \beta_n \to+\infty$. The vector $\mathbf{\bar w}_n$ is bounded in $C^{0,\alpha}(\Omega)$ by Theorem \ref{prp asymptotic k}.  Then $\mathbf{\bar w}_{n}$ solves
\[
	\begin{cases}
		- d \Delta \bar w_{i,n}  = \left(- \omega + k u_n - \bar \beta_n  \sum_{j \neq i} \bar w_{j,n} \right) \bar w_{i,n} &\text{in $\Omega$}\\
		- D \Delta u_n = \left(\lambda - \mu u_n - k \sum_{i=1}^{N_n} w_{i,n} \right)u_n &\text{in $\Omega$}\\
		\partial_\nu \bar w_{i,n} = \partial_\nu u_n = 0 &\text{on $\partial \Omega$}
	\end{cases}
\]
for $\bar \beta_n$ and $N_n$ large. In particular we have $\bar \beta_n > \hat \beta$ and $N_n > \hat N$, where $\hat \beta$ and $\hat N$ are the thresholds in Theorem \ref{thm hat}. By the assumption, $\bar w_{1,n}$ converges to a non zero limit. By the same reasoning as in \cite[Theorem 6.4]{BerestyckiZilio_PI} (see also \cite[Theorem 1.5]{BerestyckiZilio_PII} for the version of the proof in the case $N$ a priori unbounded) we get that $\bar \beta_n$ is bounded, a contradiction.

\noindent \textbf{Case 2)} There exists a sequence $\eps_n$ such that $\eps_n > 0$, $\eps_n \to 0$ and
\[
	\sup_{i=1,\dots,N_n} \|w_{i,n}\|_{L^\infty} \leq \eps_n \|H_n\|_{L^\infty}.
\]
By continuity of $H_n$, we can consider a sequence $x_n \in \overline{\Omega}$ such that
\[
	H_n(x_n) = \|H_n\|_{L^\infty}.
\]
Since, moreover, the sequence $H_n$ is uniformly bounded in $C^{0,\alpha}(\Omega)$, there exists $r > 0$ such that, for $n$ large enough,
\[
	H_n(y) > \frac12 H_n(x_n) \qquad \text{for all $y \in B_{r}(x_n) \cap \overline{\Omega}$.}
\]
We now consider the equation satisfied by $H_n$. By summing all the equations in $\mathbf{w}_n$, we find
\[
  \begin{cases}
	  - d \Delta H_{n} = \left(- \omega + k u_{n} - \beta_n H_{n}\right)H_n + \beta_n \sum_{i=1}^{N_n} w_{i,n}^2 &\text{in $\Omega$}\\
	  \partial_\nu H_n = 0 &\text{on $\partial \Omega$}.
  \end{cases}
\]
By assumption we have that for all $y \in B_{r}(x_n) \cap \overline{\Omega}$
\[
  \sum_{i=1}^{N_n} w_{i,n}^2(y) \leq \sup_{i=1,\dots,N_n} \|w_{i,n}\|_{L^\infty(\Omega)} \sum_{i=1}^{N_n} w_{i,n}(y) \leq \eps_n \|H_n\|_{L^\infty(\Omega)} H_n(y).
\]
Thus $H_n$ solves
\begin{equation}
  \begin{cases}
	  - d \Delta H_{n} \leq \left( \frac{\lambda k - \mu \omega}{\mu} - \frac12(1-2\eps_n) \beta_n  \|H_n\|_{L^\infty}  \right) H_n &\text{in $B_{r}(x_n) \cap \Omega$}\\
	  \partial_\nu H_n = 0 &\text{on $B_{r}(x_n) \cap \partial\Omega$}.
  \end{cases}
\end{equation}
If the right hand side is negative, then $\Delta H_n > 0$ in $B_r(x_n) \cap \Omega$. As $x_n$ is the maximum of $H_n$, we find a contradiction with the maximum principle if either $x_n \in \Omega$ or $x_n \in \partial \Omega$, since $\partial_\nu H_n = 0$ on $B_r(x_n) \cap \partial \Omega$. Consequently, it must be the case that
\[
	\beta_n\|H_n\|_{L^\infty} \leq \frac{2(\lambda k - \mu \omega)}{\mu(1-2\eps_n)}.
\]
Again we reach a contradiction. This concludes the proof of Lemma \ref{lem beta H bounded}.
\end{proof}

We now show a last technical estimate.

\begin{lemma}\label{lem beta a to 0 N}
For any $\eps > 0$ there exists $N_\eps \in \N$ such that if $\mathbf{v} = (\mathbf{w},u)$ is  a solution of \eqref{eqn model k same} with $\beta > 0$, $N \geq N_\eps$ and $H = \sum_{i=1}^{N} w_{i}$, then
\[
	\left\|\frac{\sum_{i=1}^N w_{i}^2}{H} \right\|_{L^\infty(\Omega)} \leq \frac{\eps}{1+\beta}.
\]
\end{lemma}
\begin{proof}
Thanks to Lemma \ref{lem all to zero} we already know that for any $\eps > 0$ there exists $N_\eps \in \N$ such that if $\beta > 0$ and $N \geq N_\eps$, then
\[
  (1+\beta)\left\|\frac{\sum_{i=1}^N w_{i}^2}{H} \right\|_{L^\infty} \leq (1+\beta)\sup_{i = 1, \dots, N} \|w_{i}\|_{L^\infty} \leq (1+\beta)\eps.
\]
Thus, if $\beta$ is bounded then the conclusion follows by taking $\eps$ sufficiently small (and thus $N_\eps$ sufficiently large).

Let us now consider the case $\beta$ large. From Lemma \ref{lem beta H bounded} we know that there exists $C>0$ such that
\[
  (1+\beta)\left\|\frac{\sum_{i=1}^N w_{i}^2}{H} \right\|_{L^\infty} \leq (1+\beta)\|H\|_{L^\infty}  \frac{ \sup_{i = 1, \dots, N} \|w_{i}\|_{L^\infty}}{\|H\|_{L^\infty}}  \leq C \frac{ \sup_{i = 1, \dots, N} \|w_{i}\|_{L^\infty}}{\|H\|_{L^\infty}}.
\]
We first show that for any $\eps > 0$ there exist $\bar \beta > 0$ and $\bar N \in \N$ such that if $\beta > \bar \beta$ and $N \geq \bar N$, then
\begin{equation}\label{eqn sup eps u}
	\sup_{i=1,\dots,N} \|w_{i}\|_{L^\infty} \leq \eps \|H\|_{L^\infty}.
\end{equation}
Again, we argue by contradiction, and assume that there exist a constant $\eta>0$ and a sequence $\mathbf{v}_n = (\mathbf{w}_n, u_n)$ of solutions with $\beta_n \to+\infty$ and $N_n \to + \infty$ such that,
\[
	\sup_{i=1,\dots,N_n} \|w_{i,n}\|_{L^\infty} > \eta \|H_{n}\|_{L^\infty}.
\]
Up to a relabeling, we can assume that $\|w_{1,n}\|_{L^\infty(\Omega)} = \sup_{i=1,\dots,N_n} \|w_{i,n}\|_{L^\infty(\Omega)}$. We introduce the scaled densities  
\[
	\mathbf{\bar w}_{n} := \frac{\mathbf{w}_{n}}{ \|H_{n}\|_{L^\infty}} \qquad \text{and} \qquad \bar H_{n} := \frac{H_{n}}{ \|H_{n}\|_{L^\infty}}.
\]
We want to show that $\| \mathbf{\bar w}_n \|_{L^\infty} \to 0$ as $n \to +\infty$ and reach a contradiction. The sequences of functions $(\mathbf{\bar w}_n, u_n)$ and $\bar H_n$ solve the systems
\begin{equation}\label{eqn sys not yet rescaled}
	\begin{cases}
		- d \Delta \bar w_{i,n}  = \left(- \omega + k u_n + \bar \beta_n \bar w_{i,n} - \bar \beta_n  \bar H_n  \right) \bar w_{i,n} &\text{in $\Omega$}\\
		- D \Delta u_n = \left(\lambda - \mu u_n - k H_n \right)u_n &\text{in $\Omega$}\\
		\partial_\nu \bar w_{i,n} = \partial_\nu u_n = 0 &\text{on $\partial \Omega$}
	\end{cases}
\end{equation}
where $\bar \beta_n := \beta_n \|H_n\|_{L^\infty}$. By Lemma \ref{lem beta H bounded}, we already know that the sequence $\bar \beta_n$ is bounded, and thus, up to striking out a subsequence, it converges to a non negative constant $\bar \beta$. By Theorem \ref{prp asymptotic k}, the functions $(\mathbf{\bar w}_n, u_n)$, $H_n$ and $\bar H_n$ are uniformly bounded in $C^{0,\alpha}(\Omega)$. Since the coefficients in each equation are bounded uniformly in the $C^{0,\alpha}$ norm, the sequence $(\mathbf{\bar w}_n, u_n)$ is also uniformly bounded in $C^{2,\alpha}(\Omega)$. Up to extracting a further subsequence, we find that any limit $(\mathbf{\bar w}, \bar u)$ and $\bar H$ solves
\begin{equation}\label{eqn sys bar rescaled}
	\begin{cases}
		- d \Delta \bar w_{i}  = \left(- \omega + k u + \bar \beta \bar w_i - \bar \beta \bar H \right) \bar w_{i} &\text{in $\Omega$}\\
		- D \Delta u = \left(\lambda - \mu u - k H \right)u &\text{in $\Omega$}\\
		\partial_\nu \bar w_i = \partial_\nu u = 0 &\text{on $\partial \Omega$}
	\end{cases}
\end{equation}
By the maximum principle, we find that, for any $i$, either $\bar w_i > 0$ or $\bar w_i \equiv 0$ in $\overline{\Omega}$. To reach a contradiction, i view of our choice of relabeling, we are going to to exclude $\bar w_1 > 0$. 

We claim that if $\bar w_1 > 0$, then $\bar w_i > 0$ for all $i$. We adopt the same strategy as in Lemma \ref{lem all to zero}. This just follows from Lemma \ref{lem no order} by arguments we have already used. 

We strengthen the claim and now show that there exists $\delta > 0$ such that
\begin{equation}\label{eqn inf of all positive}
	\inf_{x \in \Omega} \bar w_{i,n}(x) > \delta  \qquad \text{for all $i$ and $n$.}
\end{equation}
Suppose this is not the case. Then, $\inf_{x \in \Omega} \bar w_{i_n}(x) \to 0$ for a sequence $i_n \in \{1, \dots N_n\}$. Up to a relabeling, we can assume that $i_n \equiv 2$. Since the densities are uniformly bounded, we see that the sequence $\bar w_{2}$ converges in $C^{2,\alpha}$ to a solution of either \eqref{eqn sys not yet rescaled} or \eqref{eqn sys bar rescaled} that has a strictly positive maximum and is zero at some point of $\overline{\Omega}$, in contradiction with the strong maximum principle. This prove the inequality \eqref{eqn sup eps u}.

We can now complete the proof of Lemma \ref{lem beta a to 0 N}. By construction we know that $\bar H_n \leq 1$. Passing to the uniform limit we find $\bar H \leq 1$, which is incompatible with the uniform estimate \eqref{eqn inf of all positive}, since we have $N_n \to +\infty$. Lemma \ref{lem beta a to 0 N} thus follows.
\end{proof}

We are in a position to prove Proposition \ref{prp Nbeta}.

\begin{proof}[Proof of Proposition \ref{prp Nbeta}]
We first show that from any sequence $\mathbf{v}_n$ of \eqref{eqn model k same} defined for $N_n \to +\infty$ we can always extract a subsequence such that
\[
  \lim_{n \to +\infty} (1+\beta_n) \left\| H_{n} - N_n \frac{\lambda k - \mu \omega}{\mu \beta_n (N_n-1) + N_n k^2} \right\|_{C^{1,\alpha}} + \left\| u_n - \frac{\lambda \beta_n (N_n-1) + \omega k N_n }{\mu \beta_n (N_n-1) + N_n k^2}\right\|_{C^{2,\alpha}} = 0
\]
for all $\alpha \in (0,1)$. Observe that in the statement we choose a $C^{1,\alpha}$ norm of the component $H_n$ which is weaker than the $C^{2,\alpha}$ norm. We will improve the estimate at the end of the proof. We have to distinguish two cases, depending on the behavior of the sequence $\beta_n$.

\noindent \textbf{Case 1)} $\beta_n$ contains a bounded subsequence. In this case we follow closely the proof of Proposition \ref{prp beta to 0}. Up to striking out a subsequence, we can assume that $\beta_n \to \beta \geq 0$. We consider the system satisfied by $H_n$ and $u_n$,
\begin{equation}\label{eqn sys H and u n}
	\begin{cases}
		- d \Delta H_{n} = \left(- \omega + k u_{n} - \beta_n H_{n} + \beta_n \frac{\sum_{i=1}^{N_n} w_{i,n}^2}{H_n} \right) H_n &\text{in $\Omega$}\\
		- D \Delta u_{n} = \left(\lambda - \mu  u_{n} - k H_{n} \right) u_{n} &\text{in $\Omega$}\\
		\partial_\nu H_{n} = \partial_\nu u_{n} = 0 &\text{on $\partial \Omega$}.
	\end{cases}
\end{equation}
By Lemma \ref{lem beta a to 0 N} we see that all the terms in the right hand side of this system are uniformly bounded in the $L^\infty$ norm. Thus the sequence $H_n$ is uniformly bounded in $C^{1,\alpha}(\Omega)$, and therefore the sequence $u_n$ is uniformly bounded in $C^{2,\alpha}(\Omega)$. Up to striking out a subsequence, these two sequences converge to some limits $H \in C^{1,\alpha}(\Omega)$ and $u \in C^{2,\alpha}(\Omega)$ that are weak solutions of
\[
	\begin{cases}
		- d \Delta H = \left(- \omega + k u - \beta H \right) H &\text{in $\Omega$}\\
		- D \Delta u = \left(\lambda - \mu  u - k H \right) u &\text{in $\Omega$}\\
		\partial_\nu H = \partial_\nu u = 0 &\text{on $\partial \Omega$}.
	\end{cases}
\]
From Lemma \ref{lem mimura} we deduce that all solutions to the previous system are constant. More precisely, from this lemma we know that
\[
	(H,u) = (0,0), \; \text{or} \; \left(0, \frac{\lambda}{\mu}\right) \; \text{or} \; \left(\frac{\lambda k - \mu \omega}{\mu \beta + k^2}, \frac{\lambda \beta + \omega k}{\mu \beta + k^2}\right).
\]
We can exclude the first two possibilities. Indeed, if $u_n \to 0$ uniformly, then, for $n$ large enough, we have $u_n < \omega / k - \delta$ for some $\delta > 0$. But then, by Lemma \ref{lem beta a to 0 N}, $H_n$ solves
\[
	\begin{cases}
		- d \Delta H_{n} < \left(- k \delta + \eps_n \right) H_n &\text{in $\Omega$}\\
		\partial_\nu H_{n} = 0 &\text{on $\partial \Omega$}
	\end{cases}
\]
where $\eps_n > 0$ is a sequence that converges to $0$. Thus, it follows that for $n$ large enough $H_n \equiv 0$, a contradiction. We can also exclude the second possibility. Indeed, assume that $H_n \to 0$ and $u_n \to \lambda/\mu$. Let us renormalize $H_n$, by introducing the sequence $ \bar H_{n} := H_{n} / \| H_{n}\|_{L^\infty(\Omega)}$ which is bounded in $C^{1,\alpha}(\Omega)$. Up to striking out a subsequence, we can passing to the limit and find that the sequence $\bar H_n$ converges in $C^{1,\alpha}(\Omega)$ to a positive solution of
\[
	\begin{cases}
		- d \Delta \bar H  = \frac{\lambda k - \mu \omega}{\mu} \bar H &\text{in $\Omega$}\\
		\partial_\nu \bar H = 0 &\text{on $\partial \Omega$}
	\end{cases}
\]
and we find a contradiction. Thus, if the sequence $\beta_n$ contains a bounded subsequence, we finally obtain that, at least along a subsequence,
\[
  \lim_{n \to +\infty} (1+\beta_n) \left\| H_{n} - \frac{\lambda k - \mu \omega}{\mu \beta + k^2} \right\|_{C^{1,\alpha}} + \left\| u_n - \frac{\lambda \beta + \omega k}{\mu \beta + k^2} \right\|_{C^{2,\alpha}} = 0.
\]

\noindent \textbf{Case 2)} $\beta_n \to +\infty$. We consider again system \eqref{eqn sys H and u n}. From Lemma \ref{lem beta H bounded} we infer that $H_n$ converges uniformly to 0. We then consider the rescaling
\[
	\hat H_{n} := \beta_n H_{n}
\]
which, by Lemma \ref{lem beta H bounded}, is bounded in $L^\infty(\Omega)$. The functions $\hat H_n$ and $u_n$ are solutions of
\[
	\begin{cases}
		- d \Delta \hat H_{n} = \left(- \omega + k u_{n} - \hat H_{n} + \beta_n \frac{\sum_{i=1}^{N_n} w_{i,n}^2}{H_n} \right) \hat H_{n}&\text{in $\Omega$}\\
		- D \Delta u_{n} = \left(\lambda - \mu  u_{n} - k / \beta_n  \hat H_{n} \right) u_{n} &\text{in $\Omega$}\\
		\partial_\nu \hat H_{n} = \partial_\nu u_{n} = 0 &\text{on $\partial \Omega$}.
	\end{cases}
\]
Observe that by Lemma \ref{lem beta a to 0 N}, the coefficients of the right hand side of the system are all uniformly bounded in $L^\infty(\Omega)$. Once again, we find that $\hat H_n$ and $u_n$  are uniformly bounded in $C^{1,\alpha}(\Omega)$ and, up to striking out a subsequence, they converge to some limits $\hat H$ and $u$ that belong to $C^{1,\alpha}(\Omega)$. These limit functions are solutions of the system
\begin{equation}
	\begin{cases}
		- d \Delta \hat H  = \left(- \omega + k  u - \hat H \right) \hat H &\text{in $\Omega$}\\
		- D \Delta u = \left(\lambda - \mu  u \right) u &\text{in $\Omega$}\\
		\partial_\nu \hat H = \partial_\nu  u = 0 &\text{on $\partial \Omega$}.
	\end{cases}
\end{equation}
By the maximum principle, we see that $u$ is constant, and then $\hat H$ is constant as well. More precisely we have three possibilities
\[
	(\bar H,u) = (0,0), \; \text{or} \; \left(0, \frac{\lambda}{\mu}\right) \; \text{or} \;  \left(\frac{\lambda k - \mu \omega}{\mu}, \frac{\lambda}{\mu}\right).
\]
Reasoning exactly as in \textbf{Case 1)}, we can exclude the first two possibilities. Thus in the case $\beta_n \to +\infty$, we get
\[
	\lim_{n\to+\infty} \left\| \beta_n H_n - \frac{\lambda k - \mu \omega }{\mu} \right\|_{C^{1,\alpha}} + \left\| u_n - \frac{\lambda}{\mu} \right\|_{C^{2,\alpha}} = 0.
\]
and our original claim follows easily.

We now analyze the single components of $\mathbf{w}_{n}$. We consider once again the functions
\[
	\mathbf{\bar w}_{n} := \frac{\mathbf{w}_{n}}{ \|H_{n}\|_{L^\infty}}.
\]
We can exchange the elements of each vector $\mathbf{\bar w}_{n}$ in such a way that the first one has the largest $L^\infty$ norm, $\|w_{1,n}\|_{L^\infty} = \sup_{i=1,\dots,N_n}\|w_{i,n}\|_{L^\infty}$ for all $n$. Each component of $\mathbf{\bar w}_n$ solves the equation
\[
	\begin{cases}
		- d \Delta \bar w_{i,n} = \left(- \omega + k u_{n} - \beta_n H_{n} + \bar \beta_n \bar w_{i,n}\right) \bar w_{i,n} &\text{in $\Omega$}\\
		\partial_\nu \bar w_{i,n} = 0 &\text{on $\partial \Omega$}.
	\end{cases}
\]
where $\bar \beta_n := \beta_n \|H_{n}\|_{L^\infty} \to (\lambda k - \mu \omega)/\mu$. Since the coefficients of the previous equation are uniformly bounded in $C^{0,\alpha}(\Omega)$ by Theorem \ref{prp asymptotic k} and using the convergence of $\beta_n H_n$ that we have previously shown, the sequence $\bar w_{i,n}$ is uniformly bounded in $C^{2,\alpha}(\Omega)$. Passing to the limit along a suitable subsequence, we obtain that $\bar w_{1,n}$ converges in $C^{2,\alpha}(\Omega)$ to a non negative function $\bar w_1$, solution of
\[
	\begin{cases}
		- d \Delta \bar w_{1} = \frac{\lambda k - \mu \omega }{\mu} \bar w_{1}^2 &\text{in $\Omega$}\\
		\partial_\nu \bar w_{1} = 0 &\text{on $\partial \Omega$}.
	\end{cases}
\]
Since $\lambda k > \mu \omega$, we see that $\bar w_{1} \equiv 0$ and thus we find
\[
  \lim_{n \to +\infty} \frac{\sup_{i=1, \dots, N_n} \|w_{i,n}\|_{L^\infty}}{ \|H_{n}\|_{L^\infty}} = 0.
\]
We now go back to the original sequence of function $\mathbf{w}_{n}$. Each component solves
\[
	\begin{cases}
		- d \Delta w_{i,n} = \left(- \omega + k u_{n} - \beta_n H_{n} + \beta_n w_{i,n}\right) w_{i,n} = q_{i,n} w_{i,n} &\text{in $\Omega$}\\
		\partial_\nu w_{i,n} = 0 &\text{on $\partial \Omega$}.
	\end{cases}
\]
where, by the previous discussion, $\|q_{i,n}\|_{C^{0,\alpha}(\Omega)} \to 0$ for $n \to +\infty$. By considering the renormalization 
\[
  \hat w_{i,n} = \frac{ w_{i,n}}{\| w_{i,n}\|_{L^\infty} }
\]
we see that any subsequence with $N_n \to +\infty$ has to converge to the constant $1$ in $C^{2,\alpha}(\Omega)$. But then, reasoning as in Proposition \ref{prp beta to 0}, we have that
\[
  \lim_{n\to+\infty} \sup_{i,j \in \{1, \dots, N_n\}} \frac{\| w_{i,n}\|_{L^\infty}}{\| w_{j,n}\|_{L^\infty}} = 1.
\]
First, we thus find that
\[
  \lim_{n \to +\infty} \sup_{i,j \in \{1, \dots, N_n\}} \left\| \frac{ w_{i,n}}{\| w_{j,n}\|_{L^\infty} } - 1 \right\|_{C^{2,\alpha}} =  0.
\]
Second, there exists a sequence $\eps_n \to 0$ such that
\[
  (1-\eps_n)\sup_{j = 1, \dots, N_n} \| w_{j,n}\|_{L^\infty} \leq w_{i,n}(x) \leq (1+\eps_n) \inf_{j = 1, \dots, N_n} \| w_{j,n}\|_{L^\infty}
\]
for all $n \in \N$, $i = 1, \dots, N_n$ and $x \in \overline{\Omega}$. Summing up in $i$ and multiplying by $(1+\beta_n)$ we find
\[
  (1-\eps_n)N_n (1+\beta_n) \sup_{j = 1, \dots, N_n} \| w_{j,n}\|_{L^\infty} \leq (1+\beta_n) H_n \leq (1+\eps_n) N_n (1+\beta_n) \inf_{j = 1, \dots, N_n} \| w_{j,n}\|_{L^\infty}.
\]
Combining these inequalities we obtain
\[
	\lim_{n\to+\infty} N_n (1+\beta_n) \sup_{i=1,\dots,N_n} \left\| w_{i,n} - \frac{\lambda k - \mu \omega}{\mu \beta_n (N_n-1) + N_n k^2} \right\|_{C^{2,\alpha}} = 0.
\]

To conclude, it only remains to observe that the previous limit implies also the convergence of $(1+\beta_n) H_n$ to its constant limit in $C^{2,\alpha}(\Omega)$. The proof of Proposition \ref{prp Nbeta} is thus concluded.
\end{proof}

\section{Classification results of positive solutions for \texorpdfstring{$\beta$}{beta} small or \texorpdfstring{$N$}{N} large}

We now show that, when $\beta$ is small or $N$ is large, the solutions are not only close to constant, but are actually constant. Thus we derive the exact form of the solutions in these regimes. This is Theorem \ref{main thm}, which we repeat here for the readers' convenience.

\begin{theorem}\label{thm beta pos}
There exist $\bar \beta > 0$ and $\bar N \geq 1$ such that if either $0 \leq \beta \leq \bar \beta$ or $N \geq \bar N$ then the only solutions of \eqref{eqn model k same} are the constant solutions. 
\end{theorem}
We emphasize that the constant $\bar \beta$ does not depend on $N$ and, likewise, $\bar N$ does not depend on $\beta$. In particular, this result provides a generalization of \cite[Proposition 3.14]{BerestyckiZilio_PI}. Indeed, the result here is quite different from \cite[Proposition 3.14]{BerestyckiZilio_PI} in that the number of non-zero components of $\mathbf{v}$ is now a free parameter.

\begin{proof}
We assume, by contradiction, that there exists a sequence of solutions $\mathbf{v}_n = (\mathbf{w}_n, u_n)$ of \eqref{eqn model k same} with parameters $\beta_n$ and $N_n$, with either $\beta_n \to 0$ or $N_n \to +\infty$ such that $\mathbf{v}_n$ has at least one non constant component. Let $(W_n, \dots, W_n, U_n)$ with
\[
	W_n := \frac{\lambda k - \mu \omega}{\mu \beta_n (N_n-1) + N_n k^2}, \qquad U_n := \frac{\lambda \beta_n (N_n-1) + \omega k N_n }{\mu \beta_n (N_n-1) + k^2 N_n} 
\]
be the unique constant solution of system \eqref{eqn model k same} corresponding t the parameters $\beta$ and $N$. In view of Propositions \ref{prp beta to 0} and \ref{prp Nbeta}, we know that
\begin{equation}\label{eqn epsn to 0}
	\lim_{n \to \infty} \left[ (1+\beta_n) \sum_{i=1}^{N_n} \left\|w_{i,n} - W_n\right\|_{C^{2,\alpha}} + \left\|u_n - U_n\right\|_{C^{2,\alpha}} \right] = 0.
\end{equation}

We first show that the components of  $\mathbf{w}_n$ are all equal for $n$ large enough. To prove this, we make use of the structure of the system as we did in Lemma \ref{lem no order}. We assume by contradiction that there exist two sequences of indexes $i_n \neq j_n \in \{1, \dots, N_n\}$ such that the densities $w_{i_n, n} \neq w_{j_n,n}$ for all $n \in \N$. Up to a relabelling, we can assume that $i_n \equiv 1$ and $j_n \equiv 2$. 

Let us define the sequence of functions
\[
	\varphi_n := \frac{w_{1,n}-w_{2,n}}{\|w_{1,n}-w_{2,n}\|_{L^\infty(\Omega)}}.
\]
Observe that these functions are well defined, since by assumption $w_{1,n} \neq w_{2,n}$ for all $n\in \N$. Using Lemma \ref{lem no order}, we find that each $\varphi_n$ necessarily changes sign in $\Omega$. Moreover, by definition, we have $\|\varphi_n\|_{L^\infty} = 1$ for all $n \in \N$. As a result we have
\[
	\max_{\overline{\Omega}} |\varphi_n| = 1 \qquad \text{and} \qquad  \min_{\overline{\Omega}} |\varphi_n| = 0
\]
for all $n\in\N$. These functions $\varphi_n$ solve the equations
\[
	\begin{cases}
		-d \Delta \varphi_n = a_n \varphi_n &\text{in $\Omega$}\\
		\partial_\nu \varphi_n = 0 &\text{on $\partial \Omega$}.
	\end{cases}
\]
where $a_n:\Omega \to \R$ is defined as
\[
	a_n := -\omega + ku_n - \beta_n \sum_{h \geq 3} w_{h,n}.
\]
We claim that $a_n \to 0$ in $C^{2,\alpha}(\Omega)$. Indeed
\begin{multline*}
	\|a_n\|_{C^{2,\alpha}} = \left\|-\omega + k u_n - \beta_n \sum_{h \geq 3} w_{h,n} \right\|_{C^{2,\alpha}} \\
	= \left\|\underbrace{-\omega + k U_n - \beta_n (N_n-1) W_n }_{=0}+ k (u_n-U_n) - \beta_n \sum_{h \geq 3} (w_{h,n}-W_n) + \beta_n W_n \right\|_{C^{2,\alpha}}\\
	\leq k \left\| u_n-U_n\right\|_{C^{2,\alpha}} +\beta_n \sum_{h \geq 3} \left\|w_{h,n}-W_n \right\|_{C^{2,\alpha}} + \beta_n |W_n |.
\end{multline*}
The last term converges to $0$ by \eqref{eqn epsn to 0} and the definition of $W_n$. 

From the equation, we then infer that $\varphi_n$ is uniformly bounded in $W^{2,p}(\Omega)$ for all $p < \infty$. Hence, bootstrapping the regularity of $\varphi_n$, we find that $\varphi_n$ is also uniformly bounded in $C^{2,\alpha}(\Omega)$ for all $\alpha \in (0,1)$. Passing to the limit in $n$, up to striking out a subsequence, we get a function $\varphi \in C^{2,\alpha}(\Omega)$ such that $\max_{\overline{\Omega}} |\varphi| = 1$, $\min_{\overline{\Omega}} |\varphi| = 0$ and $\varphi$ is a solution of
\[
	\begin{cases}
		-d \Delta \varphi = 0 &\text{in $\Omega$}\\
		\partial_\nu \varphi = 0 &\text{on $\partial \Omega$}.
	\end{cases}
\]
This implies that $\varphi$ must be a constant, and we have thus reached a contradiction. 

Hence, for $n$ large enough, we see that
\[
	w_{i,n} \equiv w_{j,n} =: w_n \qquad \text{for all $i, j \in \{1,\dots,N_n\}$.}
\]
Exploiting this new information, from the system verified by $(\mathbf{w}_n, u_n)$ we can then extract a reduced system of two equations satisfied by $(w_n, u_n)$. It reads
\[
	\begin{cases}
		- d \Delta w_{n}  = \left(- \omega + k u_n - \beta_n (N_n-1) w_{n} \right) w_{n} &\text{in $\Omega$}\\
		- D \Delta u_n = \left(\lambda - \mu u_n - k N_n w_{n} \right)u_n &\text{in $\Omega$}\\
		\partial_\nu w_{n} = \partial_\nu u_n = 0 &\text{on $\partial \Omega$}.
	\end{cases}
\]
The pair $(w_n,u_n)$ falls under the assumptions of Lemma \ref{lem mimura}, and this entails that $w_n$ and $u_n$ are necessarily the constant solutions $W_n$ and $U_n$. The proof is thereby complete.
\end{proof}

\newpage

\appendix

\section{An estimate about compact sets}\label{app cover}
We prove a useful estimate about compact sets of $\R^n$. It can be interpreted as a continuous version of the Pigeonhole Principle.

\begin{lemma}\label{lem cover}
Let $K \subset \R^n$ be a compact set. We consider $N \in \N$ open balls $B_r(x_i)$ of centers $x_i \in K$ and radius $r > 0$. There exists a point $x \in K$ that belongs to $m_x = \sharp \{i : x \in B_r(x_i)\}$ of such balls, where $m_x$ is bounded from below by
\[
   m_x \geq N \left(\frac{\sqrt2 r}{2r + \diam K}\right)^n.
\]
Here $\diam K = \max\{ \|x-y\|: x, y \in K\}$ is the diameter of $K$.
\end{lemma}
\begin{proof}
For want of a reference, we give here a short proof. We consider the set $\overline{B_r} + K = \{x_1 + x_2 : x_1 \in \overline{B_r}, x_2 \in K\}$, of diameter $2r + \diam K$. By Jung's theorem \cite{jung}, there exists a ball $B_R$ of radius
\[
  R = (2r + \diam K) \sqrt{\frac{n}{2(n+1)}}
\]
such that $\overline{B_r} + K \subset \overline{B_R}$. Let $C_{N,r} \in \N_*$ be the largest number of balls $B_r(x_i)$ that have a common non-empty intersection
\[
  C_{N,r} = \max_{x \in K} \sharp \{i : x \in B_r(x_i)\}.
\]
We can estimate the volume of $B_R$ from below observing that any point in $B_R$ belongs to at most $C_{N,r}$ balls. We find
\[
  \sum_{i=1}^N|B_r(x_i)| \leq C_{N,r} |B_R| \implies |B_R| \geq \frac{1}{C_{N,r}} \omega_n N r^n,
\]
where $\omega_n = |B_1|$ is the volume of the ball of radius $1$. On the other hand, we have
\[
  |B_R| = \omega_n R^n = \omega_n \left(2r + \diam K \right)^n \left(\frac{n}{2(n+1)}\right)^\frac{n}{2}.
\]
Combining the two estimates, we conclude
\[
  C_{N,r} \geq N \left(\frac{r}{2r + \diam K}\right)^n \left[2\left(1+\frac{1}{n}\right)\right]^\frac{n}{2} \geq N \left(\frac{\sqrt2 r}{2r + \diam K}\right)^n. \qedhere
\]
\end{proof}

%\bibliographystyle{plain}
%\bibliography{biblio}

\end{document}